\documentclass[12pt]{amsart}
\usepackage{array}
\usepackage{amstext}
\usepackage{amsfonts}
\usepackage{amsmath}
\usepackage{filecontents}

\usepackage{amssymb}
\usepackage{tikz}
\usepackage[colorlinks=true, linkcolor=blue]{hyperref}

\usetikzlibrary{arrows, positioning}
\tikzset{Mylong/.style={text width=3.1cm, align=center}, myarr/.style={->, double equal sign distance, -implies}}

\newtheorem*{theorem*}{Theorem} % reset numberings for each chapter
\newtheorem{theorem}{Theorem}[section]

\newtheorem{lemma}[theorem]{Lemma}
\newtheorem*{lemma*}{Lemma}

\theoremstyle{definition}
\newtheorem{definition}[theorem]{Definition}
\newtheorem{example}[theorem]{Example}
\theoremstyle{remark}
\newtheorem*{remark}{Remark}

\title[infinite product of power series] {Infinite product of power series}
\author[K-Y. Chen]{Kui-Yo Chen}
\author[Z-T. Wu]{Zhong-Tang Wu}


\address{Department of Mathematics, National Cheng Kung University, No. 1, Dasyue Rd., Tainan City 70101, Taiwan}
\email{r04221001@ntu.edu.tw}
\address{Department of Applied Physics, Tunghai University, No. 1727, Sec. 4, Taiwan Boulevard, Xitun District, Taichung City 407224, Taiwan}
\email{mpcsphilo@gmail.com}

\thanks{The work of Zhong-Tang Wu is supported by the MOST research grants under the budget numbers 110-2112-M-029-003, 108-2112-M-029-006-MY3. The work of Kui-Yo Chen is partially supported by the MOST research grants under the budget numbers 111-2112-M-029-007, 111-2115-M-006-005-MY2 and 110-2115-M-006-007.
}

\date{\today}
\begin{document}

\begin{abstract}
We give an exact coefficients formula of any infinite product of power series
with constant term equal to $1$,
by using structures from partitions of integers and permutation groups.
This is an universal theorem for various of Binomial-type theorems in many sense.
In particular, we give the new formulas as the double counting of Bell polynomial, Binomial Theorem and Multinomial Theorem.
\end{abstract}

\maketitle
\addcontentsline{toc}{section}{Title}
\baselineskip17pt

\section{Introduction} % (fold)
\label{sec:introduction}
% \hskip\parindent

Consider the infinite product of power series in $\lambda$,
\begin{align*}
\prod\limits_{n=1}^{\infty}(1+\sum\limits_{k=1}^{\infty} a_{n,k}\lambda^k)
=
&\quad(1+a_{1,1}\lambda+a_{1,2}\lambda^2+a_{1,3}\lambda^3+\cdots)\\[-10pt]
&\cdot\;(1+a_{2,1}\lambda+a_{2,2}\lambda^2+a_{2,3}\lambda^3+\cdots)\\
% &\cdot\;(1+a_{3,1}\lambda+a_{3,2}\lambda^2+a_{3,3}\lambda^3+\cdots)\\
% (1+a_{4,1}\lambda+a_{4,2}\lambda^2+a_{4,3}\lambda^3+\cdots)\\
&\quad\quad\quad\quad\quad\quad
\vdots\\
&\cdot\;(1+a_{n,1}\lambda+a_{n,2}\lambda^2+a_{n,3}\lambda^3+\cdots)\\
&\quad\quad\quad\quad\quad\quad
\vdots
\end{align*}

We assume that each power series is analytic and the infinite product convergent uniformly,
% Then we can rewrite it to a power series in $\lambda$ again.
so it has the power series expansion in $\lambda$.
\begin{equation}
\label{X_k}
\prod\limits_{n=1}^{\infty}(1+\sum\limits_{k=1}^{\infty} a_{n,k}\lambda^k)
=1+\sum\limits_{k=1}^{\infty} X_k\lambda^k
=1+X_1\lambda+X_2\lambda^2+X_3\lambda^3+\cdots.
\end{equation}
Consider the polynomial ring $\mathbb{Q}[S]$ with the generating set 
\begin{equation}
\label{generating_set}
S_{\{a_{n,k}\}_{n,k=1}^\infty}:=
\left\{
	\left.
		\sum\limits_{n=1}^{\infty}\prod\limits_{j =1}^{m}a_{n,k_j}
	\right|
	m \in \mathbb{N} \text{ and } k_j \in \mathbb{N}
\right\}.
\end{equation}
We will denote the set $S_{\{a_{n,k}\}_{n,k=1}^\infty}$ by $S$ as an abbreviation if there is no confusion.
In Section \ref{sec:introduction}-\ref{sec:some_notable_applications_in_finite_product_cases},
we assume that all forms
``$\sum\limits_{n=1}^{\infty}\prod\limits_{j =1}^{m}a_{n,k_j}$''
in $S$ converge.
% For the case that some forms in $S$ are divergent, 
% % there are some trick to fix this, 
% see Section \ref{sec:convergence}.
In Section \ref{sec:convergence}, 
we provide some tricks to deal with the case that some forms in $S$ are divergent. 

In this paper,
we show that $X_k$ is in this ring $\mathbb{Q}[S]$.
Moreover, we obtain an explicit formula as our main result.
In particular, we see that each form in $S$ appears in $X_k$ for some $k$.

% Sometimes we know the information about $S$ more than the total product.
% So we can deviate the coefficients $X_k$ of infinite product from $S$.
% The coefficients are algebraically calculable if we already assign all values in $S$ to the computer, 
% but this is not the case if we directly calculate the infinite product.
Here, we provide some low-degree terms.
% We also provide code......
\begin{align*}
X_1
&= \sum\limits_{n=1}^{\infty}a_{n,1}\\
X_2
&= \sum\limits_{n=1}^{\infty}a_{n,2}
+ \frac{1}{2}\left(\sum\limits_{n=1}^{\infty}a_{n,1}\right)^2
- \frac{1}{2}\sum\limits_{n=1}^{\infty}a_{n,1}^2\\
X_3
&
= \sum\limits_{n=1}^{\infty}a_{n,3}
+ \left(\sum\limits_{n=1}^{\infty}a_{n,1}\right)\left(\sum\limits_{n=1}^{\infty}a_{n,2}\right)
- \sum\limits_{n=1}^{\infty}a_{n,1}a_{n,2}\\
&
+ \frac{1}{6}\left(\sum\limits_{n=1}^{\infty}a_{n,1}\right)^3
- \frac{1}{2}\left(\sum\limits_{n=1}^{\infty}a_{n,1}\right)
\left(\sum\limits_{n=1}^{\infty}a_{n,1}^2\right)
+ \frac{1}{3}\sum\limits_{n=1}^{\infty}a_{n,1}^3.
\end{align*}

And this leads us to the Main Theorem (Theorem \ref{finalcoeff}).

\begin{theorem*}[Main Theorem, Theorem \ref{finalcoeff}]
Given $k\in\mathbb{N}$.
The coefficient $X_k$ we described in equation(\ref{X_k}) is
$$\sum\limits_{L \vdash k}
\frac{1}{C_{stb}(L)}
\sum\limits_{\sigma\in S_{\ell(L)}}
sign(\sigma)\mathbb{A}_{\sigma,L}.$$
\end{theorem*}

Although the notations are not defined yet, see Definition \ref{nonordertuple}-\ref{AsigmaL}, we still noticed that there are some hidden structure in this formula.
The outside summation is the sums over the partitions of $k$.
The inside summation is about the permutation group $S_m$ for some $m = \ell(L) \in \mathbb{N}$. The middle one $C_{stb}(L)$ is just the size of stabilizer, see the Remark of Definition \ref{alotoffunc}.
% There are some new symbols in the theorem which we would define them in Definition \ref{nonordertuple}-\ref{AsigmaL}.
And the last symbol $\mathbb{A}_{\sigma,L}$ is defined in Definition \ref{AsigmaL} which describes the monomial terms of the polynomial in $\mathbb{Q}[S]$.

% Although the notations are not defined yet,
% we can still see that there are some hidden structure in this formula.
% The outside summation
% $$\sum\limits_{L \vdash k}$$
% sums over the partitions of $k$.
% $C_{stb}(L)$ is an positive integer.
% And the inside summation
% $$\sum\limits_{\sigma\in S_{\ell(L)}}$$
% sums over the permutation group $S_m$ for some $m=\ell(L)\in\mathbb{N}$.

There are some research that are highly related, see \cite{Bell_polynomials1,WordBell,Infinite}.

The algorithm for the Main Theorem (Theorem \ref{finalcoeff}) can be found in the GitHub, see
\url{https://github.com/kuiyochen/X_k_coeffient_calculator}.

We also provide a new expression for Ordinary Bell polynomials,
at Theorem \ref{Bell} of Section \ref{sec:bell_polynomial},
as a corollary of Main Theorem.
% We also cover Ordinary Bell polynomials in Section \ref{sec:bell_polynomial}.
% Theorem \ref{Bell} uses Main Theorem to provide a new expression for Ordinary Bell polynomials.
\begin{theorem}[Theorem \ref{Bell}]
$$\hat{B}_{k+N,N}(a_0, a_1,\cdots, a_k)
=
\sum\limits_{L \vdash k}
\frac{1}{C_{stb}(L)}
\sum\limits_{\sigma\in S_{\ell(L)}}
sign(\sigma)\mathbb{A}_{\sigma,L}(\{a_{n,k}\}_{n,k=1}^\infty),
$$
where $a_0:=1$ and 
$a_{n,k}:=
\begin{cases}
a_k  &n<N\\
0  &\text{otherwise}
\end{cases}$.
\end{theorem}

% section introduction (end)

\section{Motivation} % (fold)
\label{sec:motivation}
The motivation of this paper is inspired by calculating the eigenvalues problems of the Sturm-Liouville operators, see \cite{Physics, LS_formula}.
Let $L_\lambda$ be the Sturm-Liouville operators with an parameter $\lambda$ in $(-\epsilon, \epsilon)$ for some $\epsilon >0$.
There are some well-known statements below.
\begin{enumerate}
\item
The spectrum of $L_\lambda$ are real, discrete, bounded below and unbounded above.
\item
All elements in spectrum of $L_\lambda$ are eigenvalues of $L_\lambda$.
\item
The $n$-th eigenvalue $\theta_n$ of $L_\lambda$ are real analytic in $\lambda$ for all $n$, see \cite{POINCARIE}.
\end{enumerate}
The \textbf{Normalized Spectral Determinant} of $L_\lambda$ is defined by 
$\prod\limits_{n=1}^\infty \frac{\theta_n(\lambda)}{\theta_n(0)}$
which converges uniformly in $\lambda$ (see \cite{LS_formula}),
% So we see that the normalized spectral determinant of $L_\lambda$ is an infinite product of real analytic function in $\lambda$.
which is an infinite product of real analytic functions in $\lambda$.
On the other hand, since this infinite product converges uniformly, it is itself an analytic function as well.
It is natural to explore the relations between the Taylor coefficients of normalized spectral determinant and the coefficients of individual eigenvalues.
(Special thanks to Chuan-Tsung Chan, whose inspiration motivated us throughout this entire process.)
So after investigation, the Main Theorem (Theorem \ref{finalcoeff}) appears.
Furthermore, we found that the formula in Main Theorem has additional uses
in the case of finite product, see Section \ref{sec:some_notable_applications_in_finite_product_cases}.

% The \textbf{Spectral Zeta function} of $L_\lambda$ is defined by 
% $\zeta_{L_\lambda}(s):=\sum\limits_{i=1}^\infty \theta_i^{-s}$.
% And the values in the set ``$S$'' defined in Section \ref{sec:introduction}
% are highly related to the spectral zeta function of $L_\lambda$.
% section motivation (end)

% \section{A quantum physics system} % (fold)
% \label{sec:a_quantum_physics_system}
% We consider the simple second order ordinary differential operators
% $$L_\lambda = -\frac{d^2}{dx^2}+\frac{\lambda}{a}\delta_0$$
% on the space of continuous piecewise $C^2$-functions with domain $[-a, a]$ and boundary condition $f(-a)=f(a)=0$.
% \begin{lemma}
% \end{lemma}
% % section a_quantum_physics_system (end)

\section{The calculation of the coefficients} % (fold)
\label{sec:the_calculation_of_the_coefficients}

% Before we start, we would like to define some notations.
% For Definition \ref{nonordertuple}-\ref{alotoffunc},
% they describe some nature and simple notations.
% For Definition \ref{AL}-\ref{AsigmaL},
% they can extract the deep structures that hide in $X_k$,
% see the examples near them which would be helpful.
The notions in Definition \ref{nonordertuple}-\ref{AsigmaL} extract some deep information in the structures hidden in $X_k$.

% \begin{definition}
% \label{phiform}
% Let $\phi_s := \sum\limits_{n=1}^{\infty}\left(\frac{2}{\pi}\frac{1}{2n-1}\right)^s
% = \left(\frac{2}{\pi}\right)^s\sum\limits_{n=1}^{\infty}\frac{1}{(2n-1)^s}
% = \frac{2^s-1}{\pi^s}\zeta(s)$.
% And we formally call the rational polynomial generated by these series, ``$\phi$-form''. For instance, $\frac{4}{3}(\phi_2)^3 - 4\phi_4\phi_2 + \frac{8}{3}\phi_6$ is a $\phi$-form.
% \end{definition}
% \begin{definition}
% \label{ank}
% Define $a_{n,k}$ by the power series below.
% $$\sum\limits_{k=1}^{\infty}a_{n,k} \lambda^k
% := \left(\sum\limits_{k=1}^{\infty}\frac{\zeta_n^{(k)}}{\zeta_n^{(0)}} \lambda^k\right)^2.$$
% \end{definition}
% \begin{remark} $ $

% Note that there's a simple observation.
% For any multivariable rational polynomial $P(x_1, \cdots, x_m)$, 
% $\sum\limits_{n=1}^{\infty}P(a_{n,k_1}, a_{n,k_2}, \cdots, a_{n,k_m})$ is a $\phi$-form for fixed $k_i$.
% The reason is the number
% $\frac{\zeta_n^{(k)}}{\zeta_n^{(0)}}$ is just $p_k\left(\frac{2}{\pi}\frac{1}{2n-1}\right)$
% for some rational polynomial $p_k\in \mathbb{Q}[x]$ that is independent of $n$.
% \end{remark}

\begin{definition}
\label{nonordertuple}
Let $Y$ be an arbitrary set.
Consider the natural symmetric group action on $m$-tuple,
\begin{align*}
S_m \times Y^m &\to Y^m\\
(\sigma, (x_1, x_2, \cdots, x_m)) &\mapsto (x_{\sigma(1)}, x_{\sigma(2)}, \cdots, x_{\sigma(m)}).
\end{align*}
We say $L$ is a \textbf{non-ordered $m$-tuple}
if $L$ is an element of the orbit space of this action.
For $t=(x_1, x_2, \cdots, x_m)$, we denote the corresponding orbit by the brackets $[x_1, x_2, \cdots, x_m]$.
For instance,
$[1, 1, 3, 5, 3]=[5, 3, 3, 1, 1]=[1, 1, 3, 3, 5].$
Intuitively speaking, the non-ordered tuple is the tuple that forget its ordering.
\end{definition}

Now, we would like to define various of functions or symbols with a little bit of confuse notation.
\begin{definition} $ $
\label{alotoffunc}
\begin{enumerate}
\item
The \textit{ordered forget function} $F_O((x_1, x_2, \cdots, x_m)):=[x_1, x_2, \cdots, x_m]$.
\item
The \textit{tuple forget function}\\
$F_T([x_1, x_2, \cdots, x_m]):=\{x_i|i=1, \cdots , m\}$.
% Also, $F_T((x_1, x_2, \cdots, x_m)):=\{x_i|i=1, \cdots , m\}$.
\item
\textit{Partitions of integers}. Let $L = [x_1, x_2, \cdots, x_m]$ with $x_i$ being positive integers.
We say $L$ is a partition of $k$, denote by
$$L\vdash k,$$
if $\sum\limits_{i=1}^m x_i = k$.
\item
The \textit{stabilizer counting function}. For $L = [x_1, x_2, \cdots, x_m]$,
$$C_{stb}(L):=\prod\limits_{x\in F_T(L)} (\#\{i| x_i=x\})!.$$
For instance,
$C_{stb}([1, 1, 1, 1, 3, 3, 5])=4!\cdot2!\cdot1!=24\cdot 2\cdot 1=48.$
The subscript ``stb'' means the ``stabilizer'', see the remark below.
\item
The \textit{length function}. For $L = [x_1, x_2, \cdots, x_m]$,
$$\ell(L)=m.$$
For instance,
$\ell([1, 1, 2, 2, 3])=5.$
\end{enumerate}
\end{definition}
\begin{remark} $ $

Given a non-ordered tuple $L$ with length $m$.
Consider a $S_m$ group action on the set $F_O^{-1}(\{L\})$ (the pre-image of $F_O$) by permuting components. The number $C_{stb}(L)$ is in fact the cardinality of the stabilizer at any $t\in F_O^{-1}(\{L\})$. For example, $L = [1, 1, 1, 1, 3, 3, 5], t = (1, 1, 1, 1, 3, 3, 5)\in F_O^{-1}(\{L\})$.
The stabilizer $G_t$ is a subgroup in $S_m$ that is isomorphic to $S_4\times S_2\times S_1$,
the cardinality is $4!\cdot2!\cdot1!$.
% thus it has cardinality $4!\cdot2!\cdot1!$.
\end{remark}

\begin{definition}
\label{AL}
For a given non-ordered tuple of integers $L=[k_1, k_2, \cdots, k_m]$, we define
$$A_L(\{a_{n,k}\}_{n,k=1}^\infty)=A_{[k_1, k_2, \cdots, k_m]}(\{a_{n,k}\}_{n,k=1}^\infty):=
\sum\limits_{n=1}^{\infty}\prod\limits_{j=1}^{m}a_{n,k_j}.$$
And for given non-ordered tuples $L_1, L_2, \cdots L_m$, we define
$$A_{[L_1, L_2, \cdots, L_m]}(\{a_{n,k}\}_{n,k=1}^\infty):=
\prod\limits_{j=1}^{m}A_{L_j}(\{a_{n,k}\}_{n,k=1}^\infty).$$

We will denote $A_L(\{a_{n,k}\}_{n,k=1}^\infty), A_{[L_1, L_2, \cdots, L_m]}(\{a_{n,k}\}_{n,k=1}^\infty)$
by $A_L, A_{[L_1, L_2, \cdots, L_m]}$ as an abbreviation if there is no confusion.

For instance, let $L_1=[1, 1, 3], L_2=[2, 4], L_3=[5]$. Then
\begin{align*}
A_{[L_1, L_2, L_3]}&=A_{L_1}A_{L_2}A_{L_3}\\
&=A_{[1,1,3]}A_{[2,4]}A_{[5]}\\
&=
\left(\sum\limits_{n=1}^{\infty}a_{n,1}a_{n,1}a_{n,3}\right)
\left(\sum\limits_{n=1}^{\infty}a_{n,2}a_{n,4}\right)
\left(\sum\limits_{n=1}^{\infty}a_{n,5}\right)\\
&=
\left(\sum\limits_{n=1}^{\infty}a_{n,1}^2a_{n,3}\right)
\left(\sum\limits_{n=1}^{\infty}a_{n,2}a_{n,4}\right)
\left(\sum\limits_{n=1}^{\infty}a_{n,5}\right).
\end{align*}
So, every $A_{[L_1, L_2, \cdots, L_m]}$ is in $\mathbb{Q}[S]$
since it is just a product of the forms in $S$.
\end{definition}

\begin{definition}
\label{AsigmaL}
For given fixed non-ordered tuple $L=[k_1, k_2, \cdots, k_m]$
with length $m$ and $k_j$ being non-decreasing
and $\sigma\in S_m$.
Since every permutation can be expressed as disjoint cycles,
$$\sigma = (c_{1,1}\cdots c_{1, q_1})(c_{2,1}\cdots c_{2, q_2})\cdots (c_{m',1}\cdots c_{m', q_{m'}}),$$
we define
$$\mathbb{A}_{\sigma,L}(\{a_{n,k}\}_{n,k=1}^\infty):=
A_{[[k_{c_{1,1}}\cdots k_{c_{1, q_1}}],[k_{c_{2,1}}\cdots k_{c_{2, q_2}}],\cdots ,
[k_{c_{m',1}}\cdots k_{c_{m', q_{m'}}}]]}(\{a_{n,k}\}_{n,k=1}^\infty).
$$

We will denote $\mathbb{A}_{\sigma,L}(\{a_{n,k}\}_{n,k=1}^\infty)$
by $\mathbb{A}_{\sigma,L}$ as an abbreviation if there is no confusion.

\begin{remark}
It is not required that $k_j$ be chosen to be non-decreasing,
which is irrelevant in the final purpose (The Main Theorem \ref{finalcoeff}).
However, in order to make the definition of $\mathbb{A}_{\sigma,L}$ clear,
we still assume that $k_j$ are non-decreasing.
\end{remark}
% Notice that for cycles with length $1$, we need to write it down in this section!!
% Here, for cycles with length $1$, we still need to keep them in the expression of disjoint cycles.
% For length-$1$ cycles, we still keep them as such expressions mentioned above.
% 跟一班的寫法不同
% Contrary to the usual convention, the classical expression of permutations omits the length-one-cycles, 
% but we need to keep them in this paper.
Notice that in the classical cycle expressions of permutations, we omit all length-one-cycles. 
However in this paper, we will keep them for all cycle expressions of permutations.
For instance, let $L=[k_1, \cdots, k_6]$ and $\sigma = (1 2 4)(3 6)(5)\in S_6$.
$$\mathbb{A}_{\sigma,L}=
A_{[[k_1 k_2 k_4],[k_3 k_6],[k_5]]}
=
\left(\sum\limits_{n=1}^{\infty}a_{n,k_1}a_{n,k_2}a_{n,k_4}\right)
\left(\sum\limits_{n=1}^{\infty}a_{n,k_3}a_{n,k_6}\right)
\left(\sum\limits_{n=1}^{\infty}a_{n,k_5}\right).
$$
\end{definition}

Next, we go into the detail of the proof of main theorem.
Especially, Lemma \ref{unequalsum} provides the most important algorithm for calculation of $X_k$.

\begin{lemma}
\label{unequalsum}
We have the following equation.
$$\sum\limits_{\substack{n_1, \cdots, n_m\\ 
\text{are pairwise distinct}}}
\prod\limits_{j=1}^{m}a_{n_j,k_j} = 
\sum\limits_{\sigma\in S_m}sign(\sigma)\mathbb{A}_{\sigma, L},$$
for $L = [k_1,k_2,\cdots,k_m]$.
\end{lemma}
% \newpage
\begin{example}
For, $L=[1,2,3]$.\\
$\sum\limits_{\substack{n_1, n_2, n_3=1\\ 
n_1, n_2, n_3\text{ are pairwise distinct}}}^{\infty}
a_{n_1,1}a_{n_2,2}a_{n_3,3}=$
% \vspace{-0.4cm}
$$
\begin{array}{lll}
	\begin{array}{lllll}
	\raisebox{0.52cm}{ }\raisebox{-0.52cm}{ } &\mathbb{A}_{(1)(2)(3),L} \\
	\raisebox{0.52cm}{ }\raisebox{-0.52cm}{ }+&\mathbb{A}_{(123),L}     \\
	\raisebox{0.52cm}{ }\raisebox{-0.52cm}{ }+&\mathbb{A}_{(132),L}     \\
	\raisebox{0.52cm}{ }\raisebox{-0.52cm}{ }-&\mathbb{A}_{(1)(23),L}   \\
	\raisebox{0.52cm}{ }\raisebox{-0.52cm}{ }-&\mathbb{A}_{(12)(3),L}   \\
	\raisebox{0.52cm}{ }\raisebox{-0.52cm}{ }-&\mathbb{A}_{(13)(2),L} 
	\end{array}
	&=&
	\begin{array}{lllll}
	&  &\left(\sum\limits_{n=1}^{\infty}a_{n,1}\right)
		\left(\sum\limits_{n=1}^{\infty}a_{n,2}\right)
		\left(\sum\limits_{n=1}^{\infty}a_{n,3}\right)\\
	& +&\left(\sum\limits_{n=1}^{\infty}a_{n,1}a_{n,2}a_{n,3}\right)\\
	& +&\left(\sum\limits_{n=1}^{\infty}a_{n,1}a_{n,3}a_{n,2}\right)\\
	& -&\left(\sum\limits_{n=1}^{\infty}a_{n,1}\right)
		\left(\sum\limits_{n=1}^{\infty}a_{n,2}a_{n,3}\right)\\
	& -&\left(\sum\limits_{n=1}^{\infty}a_{n,1}a_{n,2}\right)
		\left(\sum\limits_{n=1}^{\infty}a_{n,3}\right)\\
	& -&\left(\sum\limits_{n=1}^{\infty}a_{n,1}a_{n,3}\right)
		\left(\sum\limits_{n=1}^{\infty}a_{n,2}\right)
	\end{array}
\end{array}
$$

% $$
% \begin{array}{lllll}
% &a_{(1)(2)(3),L} &  & &\left(\sum\limits_{n=1}^{\infty}a_{n,1}\right)
% \left(\sum\limits_{n=1}^{\infty}a_{n,2}\right)
% \left(\sum\limits_{n=1}^{\infty}a_{n,3}\right)\\
% +&a_{(123),L} &  & +&\left(\sum\limits_{n=1}^{\infty}a_{n,1}a_{n,2}a_{n,3}\right)\\
% +&a_{(132),L} &= & +&\left(\sum\limits_{n=1}^{\infty}a_{n,1}a_{n,3}a_{n,2}\right)\\
% -&a_{(1)(23),L} &  & -&\left(\sum\limits_{n=1}^{\infty}a_{n,1}\right)\left(\sum\limits_{n=1}^{\infty}a_{n,2}a_{n,3}\right)\\
% -&a_{(12)(3),L} &  & -&\left(\sum\limits_{n=1}^{\infty}a_{n,1}a_{n,2}\right)\left(\sum\limits_{n=1}^{\infty}a_{n,3}\right)\\
% -&a_{(13)(2),L} &  & -&\left(\sum\limits_{n=1}^{\infty}a_{n,1}a_{n,3}\right)\left(\sum\limits_{n=1}^{\infty}a_{n,2}\right)
% \end{array}
% $$

% \begin{align*}
% \sum\limits_{\substack{i_1, i_2, i_3=1\\ 
% i_1, i_2, i_3\text{ are pairwise distinct}}}^{\infty}
% a_{i_1,1}a_{i_2,2}a_{i_3,3}
% &=\left(\sum\limits_{n=1}^{\infty}a_{n,1}\right)
% \left(\sum\limits_{n=1}^{\infty}a_{n,2}\right)
% \left(\sum\limits_{n=1}^{\infty}a_{n,3}\right)\\
% &+\left(\sum\limits_{n=1}^{\infty}a_{n,1}a_{n,2}a_{n,3}\right)\\
% &+\left(\sum\limits_{n=1}^{\infty}a_{n,1}a_{n,3}a_{n,2}\right)\\
% &-\left(\sum\limits_{n=1}^{\infty}a_{n,1}\right)\left(\sum\limits_{n=1}^{\infty}a_{n,2}a_{n,3}\right)\\
% &-\left(\sum\limits_{n=1}^{\infty}a_{n,1}a_{n,2}\right)\left(\sum\limits_{n=1}^{\infty}a_{n,3}\right)\\
% &-\left(\sum\limits_{n=1}^{\infty}a_{n,1}a_{n,3}\right)\left(\sum\limits_{n=1}^{\infty}a_{n,2}\right).
% \end{align*}
\end{example}
\begin{example}
For, $L=[1,1,1,1,1]$.
\begin{align*}
&\sum\limits_{\substack{n_1, n_2, n_3, n_4, n_5=1\\ 
n_1, n_2, n_3, n_4, n_5\text{ are pairwise distinct}}}^{\infty}
a_{n_1,1}a_{n_2,1}a_{n_3,1}a_{n_4,1}a_{n_5,1}\\
&=1\left(\sum\limits_{n=1}^{\infty}a_{n,1}\right)^5
-10\left(\sum\limits_{n=1}^{\infty}a_{n,1}\right)^3\left(\sum\limits_{n=1}^{\infty}a_{n,1}^2\right)\\
&+20\left(\sum\limits_{n=1}^{\infty}a_{n,1}\right)^2
\left(\sum\limits_{n=1}^{\infty}a_{n,1}^3\right)
+15\left(\sum\limits_{n=1}^{\infty}a_{n,1}\right)
\left(\sum\limits_{n=1}^{\infty}a_{n,1}^2\right)^2\\
&-30\left(\sum\limits_{n=1}^{\infty}a_{n,1}\right)\left(\sum\limits_{n=1}^{\infty}a_{n,1}^4\right)
-20\left(\sum\limits_{n=1}^{\infty}a_{n,1}^2\right)\left(\sum\limits_{n=1}^{\infty}a_{n,1}^3\right)
+24\left(\sum\limits_{n=1}^{\infty}a_{n,1}^5\right).
\end{align*}
If $\sigma,\sigma'$ have same cycle type,
then
$\mathbb{A}_{\sigma,[1,1,1,1,1]}=\mathbb{A}_{\sigma',[1,1,1,1,1]}$.
Notice that $\sigma,\sigma'$ have same cycle type if and only if
the conjugacy classes $[\sigma]$ and $[\sigma']$ are equal, 
so the coefficients $1,10,20,15,30,20,24$
are the cardinality of the conjugacy classes
$$[e],[(12)],[(123)],[(12)(34)],[(1234)],[(12)(345)],[(12345)].$$
\end{example}

\begin{proof}[proof of lemma \ref{unequalsum}]
We will prove it by induction on $m$.
For the base case $m=1$, it can be verified directly from the definition.
For the inductive step $m>1$.
\begin{align*}
\sum\limits_{\substack{n_1, \cdots, n_m\\ 
\text{are pairwise distinct}}}
\prod\limits_{j=1}^{m}a_{n_j,k_j}
&=
\sum\limits_{\substack{n_1, \cdots, n_{m-1}\\ 
\text{are pairwise distinct}}}
\sum\limits_{n_m}
\prod\limits_{j=1}^{m}a_{n_j,k_j}\\
&-
\sum\limits_{\substack{n_1, \cdots, n_{m-1}\\ 
\text{are pairwise distinct}\\
n_1=n_m
}}
\prod\limits_{j=1}^{m}a_{n_j,k_j}\\
&-
\sum\limits_{\substack{n_1, \cdots, n_{m-1}\\ 
\text{are pairwise distinct}\\
n_2=n_m
}}
\prod\limits_{j=1}^{m}a_{n_j,k_j}\\
&-\cdots\\
&-
\sum\limits_{\substack{n_1, \cdots, n_{m-1}\\ 
\text{are pairwise distinct}\\
n_{m-1}=n_m
}}
\prod\limits_{j=1}^{m}a_{n_j,k_j}\\
&=
\left(
\sum\limits_{\substack{n_1, \cdots, n_{m-1}\\ 
\text{are pairwise distinct}}}
\prod\limits_{j=1}^{m-1}a_{n_j,k_j}
\right)
\left(\sum\limits_{n_m}a_{n_m,k_m}\right)\\
&-
\sum\limits_{\alpha=1}^{m-1}
\sum\limits_{\substack{n_1, \cdots, n_{m-1}\\ 
\text{are pairwise distinct}\\
n_\alpha =n_m
}}
\prod\limits_{j=1}^{m}a_{n_j,k_j}.
\end{align*}

In order to apply induction hypothesis, define
$$\tilde{a}_{n_j,k_j}^{(\alpha)} := 
\begin{cases}
a_{n_j,k_j}a_{n_m,k_m}  &j=\alpha\\
a_{n_j,k_j}  &j\neq \alpha,
\end{cases}$$
for $\alpha=1,\cdots,m-1$, 
and define
$$\tilde{\mathbb{A}}^{(\alpha)}_{\sigma,[k_1,\cdots,k_{m-1}]}:= 
\mathbb{A}_{\sigma,[k_1,\cdots,k_{m-1}]}(\{\tilde{a}_{n,k}^{(\alpha)}\}_{n,k=1}^\infty).$$
% Similarly, define the symbol
% `$\tilde{\mathbb{A}}^{(\alpha)}_{\sigma,[k_1,\cdots,k_{m-1}]}$'
% by replacing `$a$' in the definition of
% `$\mathbb{A}_{\sigma,[k_1,\cdots,k_{m-1}]}$'
% by `$\tilde{a}^{(\alpha)}$'.
% That is,
% for $\sigma \in S_{m-1}$ express as disjoint cycles
% $$(c_{1,1}\cdots c_{1, q_1})(c_{2,1}\cdots c_{2, q_2})\cdots (c_{m',1}\cdots c_{m', q_{m'}}),$$
% we define
% \begin{align*}
% \tilde{\mathbb{A}}^{(\alpha)}_{\sigma,[k_1,\cdots,k_{m-1}]}&:=
% \tilde{A}^{(\alpha)}_
% {[[k_{c_{1,1}}\cdots k_{c_{1, q_1}}],[k_{c_{2,1}}\cdots k_{c_{2, q_2}}],\cdots ,
% [k_{c_{m',1}}\cdots k_{c_{m', q_{m'}}}]]}\\
% &:=
% \prod\limits_{j=1}^{m'}\tilde{A}^{(\alpha)}_{[k_{c_{j,1}}\cdots k_{c_{j, q_j}}]}\\
% &:=
% \prod\limits_{j=1}^{m'}
% \left(\sum\limits_{n=1}^{\infty}\prod\limits_{r=1}^{q_j}\tilde{a}^{(\alpha)}_{n,k_{c_{j,r}}}\right).
% \end{align*}

% Also, define a partition of $S_m$ to $m$ subsets with same size as follows.
Define a partition of the set $S_m$.
$$D_\alpha:=\left\{
\left.
\sigma\in S_m
\right|
\sigma(\alpha)=m, \text{ i.e. $\sigma$ sends $\alpha$ to $m$}
\right\}.
$$
So $S_m = D_m \dot\cup \left(\dot\bigcup_{\alpha=1}^{m-1} D_\alpha\right)$.
For $\alpha=m$, $D_m$ is a subgroup of $S_m$ that is isomorphic to $S_{m-1}$.
% On the other hand, there is some relation between $S_{m-1}$ and $D_\alpha$'s.
For $\alpha=1, \cdots, m-1$,
$S_{m-1}$ and $D_\alpha$ are bijective (but not isomorphic as groups)
% there are bijections between $S_{m-1}$ and $D_\alpha$'s.
% Then we have $S_m = D_m \dot\cup \left(\dot\bigcup_{\alpha=1}^{m-1} D_\alpha\right)$.
, and these bijections are defined as follows.
Define $\iota_\alpha:S_{m-1}\to D_\alpha$ by 
$$
\iota_\alpha(\sigma)(i) = 
\begin{cases}
\sigma(\alpha)  &i=m\\
m  &i=\alpha\\
\sigma(i)  &\text{otherwise}.
\end{cases}
$$
Intuitively speaking, $\iota_\alpha$ insert $m$ behind $\alpha$.
For example, $m=7$,
\begin{align*}
\iota_3((154)(26)(3))&=(154)(26)(37)\\
\iota_3((13)(256)(4))&=(137)(256)(4)\\
\iota_3((1235)(64))&=(12375)(64)\\
\iota_3(e)=\iota_3((1)(2)(3)(4)(5)(6))&=(1)(2)(37)(4)(5)(6)=(37).
\end{align*}

We may verify the following things which we would like to left it as an exercise to the reader.
For $\alpha=1, \cdots, m-1$ and $\sigma\in S_{m-1}$,
\begin{enumerate}
% \item
% $\iota_\alpha$ is bijective.
\item
% For $\sigma\in S_{m-1}$,
$\tilde{\mathbb{A}}^{(\alpha)}_{\sigma,[k_1,\cdots,k_{m-1}]}=\mathbb{A}_{\iota_\alpha(\sigma),[k_1,\cdots,k_{m-1},k_m]}.$
\item
% For $\sigma\in S_{m-1}$,
$sign(\sigma) = -sign(\iota_\alpha(\sigma)).$
\end{enumerate}

So we find that
\begin{align*}
&
\left(
\sum\limits_{\substack{n_1, \cdots, n_{m-1}\\ 
\text{are pairwise distinct}}}
\prod\limits_{j=1}^{m-1}a_{n_j,k_j}
\right)
\left(\sum\limits_{n_m}a_{n_m,k_m}\right)
-
\sum\limits_{\alpha=1}^{m-1}
\sum\limits_{\substack{n_1, \cdots, n_{m-1}\\ 
\text{are pairwise distinct}\\
n_\alpha =n_m
}}
\prod\limits_{j=1}^{m}a_{n_j,k_j}\\
&=
\left(
\sum\limits_{\sigma\in S_{m-1}}
sign(\sigma)\mathbb{A}_{\sigma,[k_1,\cdots,k_{m-1}]}
\right)
\left(\sum\limits_{n_m}a_{n_m,k_m}\right)
-
\sum\limits_{\alpha=1}^{m-1}
\sum\limits_{\substack{n_1, \cdots, n_{m-1}\\ 
\text{are pairwise distinct}
}}
\prod\limits_{j=1}^{m-1}\tilde{a}_{n_j,k_j}^{(\alpha)}\\
&=
\sum\limits_{\sigma\in D_m}
sign(\sigma)\mathbb{A}_{\sigma,[k_1,\cdots,k_{m-1},k_m]}
-
\sum\limits_{\alpha=1}^{m-1}
\sum\limits_{\sigma\in S_{m-1}}
sign(\sigma)\tilde{\mathbb{A}}^{(\alpha)}_{\sigma,[k_1,\cdots,k_{m-1}]}\\
&=
\sum\limits_{\sigma\in D_m}
sign(\sigma)\mathbb{A}_{\sigma,[k_1,\cdots,k_{m-1},k_m]}
+
\sum\limits_{\alpha=1}^{m-1}
\sum\limits_{\iota_\alpha(\sigma)\in \iota_\alpha(S_{m-1})}
sign(\iota_\alpha(\sigma))\mathbb{A}_{\iota_\alpha(\sigma),[k_1,\cdots,k_{m-1},k_m]}\\
&=
\sum\limits_{\sigma\in D_m}
sign(\sigma)\mathbb{A}_{\sigma,[k_1,\cdots,k_{m-1},k_m]}
+
\sum\limits_{\alpha=1}^{m-1}
\sum\limits_{\sigma\in D_\alpha}
sign(\sigma)\mathbb{A}_{\sigma,[k_1,\cdots,k_{m-1},k_m]}\\
&=
\sum\limits_{\sigma\in S_m}
sign(\sigma)\mathbb{A}_{\sigma,L}.
\end{align*}
Thus, the equation holds.
\end{proof}

\begin{lemma}
\label{permutesum}
Given a non-ordered tuple $L$ with length $m$.
We have the following equation.
$$\sum\limits_{(k_1, \cdots, k_m)\in F_O^{-1}(\{L\})}
\sum\limits_{1\le n_1<n_2<\cdots<n_m}\prod\limits_{j=1}^{m}a_{n_j,k_j}=
\frac{1}{C_{stb}(L)}
\sum\limits_{\sigma\in S_m}
sign(\sigma)\mathbb{A}_{\sigma,L}.$$
% can be represent as a rational polynomial of the forms
% $\sum\limits_{n=1}^{\infty}\prod\limits_{\alpha =1}^{m'}a_{n,k_\alpha' }$.
% In particular, it is a $\phi$-form.
\end{lemma}
\begin{proof}
Fix a $t_0 = (k_1^{(0)}, k_2^{(0)}, \cdots, k_m^{(0)})\in F_O^{-1}(\{L\})$.
Since $S_m$ acts transitively on $F_O^{-1}(\{L\})$,
for any $t = (k_1, k_2, \cdots, k_m)\in F_O^{-1}(\{L\})$, 
$t = \sigma \cdot t_0$ for some $\sigma \in S_n$.
That is, 
$t = (k_1, k_2, \cdots, k_m) = (k_{\sigma(1)}^{(0)}, k_{\sigma(2)}^{(0)}, \cdots, k_{\sigma(m)}^{(0)})$. Also, the number of all possible $\sigma$ such that $t = \sigma \cdot t_0$ is $C_{stb}(L)$.

Then
\begin{align*}
C_{stb}(L)
\left(
\sum\limits_{(k_1, \cdots, k_m)\in F_O^{-1}(\{L\})}
\sum\limits_{1\le n_1<\cdots<n_m}\prod\limits_{j=1}^{m}a_{n_j,k_j}
\right)
&=
\sum\limits_{\sigma \in S_m}
\sum\limits_{1\le n_1<\cdots<n_m}\prod\limits_{j=1}^{m}a_{n_j,k_{\sigma(j)}^{(0)}}\\
&=
\sum\limits_{\sigma \in S_m}
\sum\limits_{1\le n_1<\cdots<n_m}\prod\limits_{j=1}^{m}a_{n_{\sigma^{-1}(j)},k_j^{(0)}}\\
% \text{($n_j':=n_{\sigma^{-1}(j)}$, 
% so $n_{\sigma(j)}' = n_j$)  }
&=
\sum\limits_{\sigma \in S_m}
\sum\limits_{1\le n_{\sigma(1)}'<\cdots<n_{\sigma(m)}'}\prod\limits_{j=1}^{m}a_{n_j',k_j^{(0)}}\\
&=
\sum\limits_{\substack{n_1', \cdots, n_m'\\ \text{are pairwise distinct}}}\prod\limits_{j=1}^{m}a_{n_j',k_j^{(0)}}.
\end{align*}
By Lemma \ref{unequalsum}, the proof is completed.
\end{proof}
\begin{theorem}
\label{finalcoeff}
Given $k\in\mathbb{N}$.
The coefficient $X_k$ we described in equation(\ref{X_k}) is
$$\sum\limits_{L \vdash k}
\frac{1}{C_{stb}(L)}
\sum\limits_{\sigma\in S_{\ell(L)}}
% \sum\limits_{\substack{\sigma\in S_m\\ m=length L}}
sign(\sigma)\mathbb{A}_{\sigma,L}.$$
% The coefficient $X_k$ we described in equation ... is a $\phi$-form.
% The coefficient $X_k$ we describe in equation ... is
% $$
% % \sum\limits_{\substack{\sum\limits_{j =1}^{m}k_j=n\\ 1\le k_j}}
% \sum\limits_{\substack{\Sigma_t=n,\\ t=(k_1, \cdots, k_m)\\ 1\le k_j}}
% % \sum\limits_{\substack{\Sigma_t=n,\\ \text{t is a tuple.}}}
% % \sum\limits_{(k_1, \cdots, k_m)\in F_O^{-1}(\{L\})}
% % \sum\limits_{\substack{1\le i_1<i_2<\cdots<i_m,\\ \sum\limits_{j =1}^{m}i_j=n}}
% \sum\limits_{1\le i_1<i_2<\cdots<i_m}
% \prod\limits_{j=1}^{m}a_{i_j,k_j}.$$
% And it is a $\phi$-form.
\end{theorem}
\begin{proof}
Recall that 
$\prod\limits_{n=1}^{\infty}(1+\sum\limits_{k=1}^{\infty} a_{n,k}\lambda^k)=
1+\sum\limits_{k=1}^{\infty} X_k\lambda^k$.
From this equation, we directly calculate $X_k$,
$$X_k=
\sum\limits_{\substack{k_1 + \cdots + k_m=k\\ 1\le k_j}}
\sum\limits_{1\le n_1<n_2<\cdots<n_m}
\prod\limits_{j=1}^{m}a_{n_j,k_j}.$$
Observe that
$$\sum\limits_{\substack{k_1 + \cdots + k_m=k\\ 1\le k_j}}
\sum\limits_{1\le n_1<n_2<\cdots<n_m}
\prod\limits_{j=1}^{m}a_{n_j,k_j}$$
$$
=
\sum\limits_{L \vdash k}
\left(
\sum\limits_{(k_1, \cdots, k_m)\in F_O^{-1}(\{L\})}
\sum\limits_{1\le n_1<n_2<\cdots<n_m}
\prod\limits_{j=1}^{m}a_{n_j,k_j}
\right).$$
By Lemma \ref{permutesum},
the proof is done.
% By Lemma \ref{permutesum}, $X_k$ is a sum of $\phi$-forms which is also a $\phi$-form.
\end{proof}

% section the_calculation_of_the_coefficients (end)

\section{Some notable applications in finite-product cases} % (fold)
\label{sec:some_notable_applications_in_finite_product_cases}

Consider the form in Equation (\ref{X_k}) whose product part is finite, that is, 
$$\prod\limits_{n=1}^{N}(1+\sum\limits_{k=1}^{\infty} a_{n,k}\lambda^k)
=1+\sum\limits_{k=1}^{\infty} X_k\lambda^k$$
for some $N\in\mathbb{N}$.
This is equivalent to the infinite product in Equation (\ref{X_k}) with $a_{n,k}:=0$ for all $k\in\mathbb{N}$ and all $n > N$.
That is,
\begin{align*}
\prod\limits_{n=1}^{\infty}(1+\sum\limits_{k=1}^{\infty} a_{n,k}\lambda^k)
=&\quad(1+a_{1,1}\lambda+a_{1,2}\lambda^2+a_{1,3}\lambda^3+\cdots)\\[-10pt]
&\cdot\;(1+a_{2,1}\lambda+a_{2,2}\lambda^2+a_{2,3}\lambda^3+\cdots)\\
% &(1+a_{3,1}\lambda+a_{3,2}\lambda^2+a_{3,3}\lambda^3+\cdots)\\
% &(1+a_{4,1}\lambda+a_{4,2}\lambda^2+a_{4,3}\lambda^3+\cdots)\\
&\quad\quad\quad\quad\quad\quad
\vdots\\
&\cdot\;(1+a_{N,1}\lambda+a_{N,2}\lambda^2+a_{N,3}\lambda^3+\cdots)\\
&\cdot\;(1+0\lambda+0\lambda^2+0\lambda^3+\cdots)\\
&\cdot\;(1+0\lambda+0\lambda^2+0\lambda^3+\cdots)\\
&\quad\quad\quad\quad\quad\quad
\vdots\\
=&\prod\limits_{n=1}^{N}(1+\sum\limits_{k=1}^{\infty} a_{n,k}\lambda^k)=1+\sum\limits_{k=1}^{\infty} X_k\lambda^k.
\end{align*}

% In this case,
Since this is a finite product,
% the Main Theorem, Theorem \ref{finalcoeff}, works without the need for consideration of convergence.
we can apply the Main Theorem(Theorem \ref{finalcoeff}) regardless of convergences of the forms in $S$,
then the generating set $S$ (Equation (\ref{generating_set})) would be
$$
S:=
\left\{
	\left.
		\sum\limits_{n=1}^{N}\prod\limits_{\alpha =1}^{m}a_{n,k_\alpha }
	\right|
	m \in \mathbb{N} \text{ and } k_\alpha \in \mathbb{N}
\right\}.
$$
Hence, the Main Theorem is not only a formula for ``infinite'' product of power series,
but also a formula for the ``finite'' product of power series.

We provide some low-degree terms by applying the Main Theorem,
\begin{align*}
X_1
&= \mathbb{A}_{e,[1]}= \sum\limits_{n=1}^{N}a_{n,1}\\
X_2
&= \mathbb{A}_{e,[2]} + \frac{1}{2}\left(\mathbb{A}_{e,[1,1]}-\mathbb{A}_{(12),[1,1]}\right)\\
&= \sum\limits_{n=1}^{N}a_{n,2}
+ \frac{1}{2}\left(\sum\limits_{n=1}^{N}a_{n,1}\right)^2
- \frac{1}{2}\sum\limits_{n=1}^{N}a_{n,1}^2\\
X_3
&
= \mathbb{A}_{e,[3]} + \left(\mathbb{A}_{e,[1,2]}-\mathbb{A}_{(12),[1,2]}\right)\\
&
+ \frac{1}{3!}\left(
\mathbb{A}_{e,[1,1,1]}+\mathbb{A}_{(123),[1,1,1]}+\mathbb{A}_{(132),[1,1,1]}
-\mathbb{A}_{(12),[1,1,1]}-\mathbb{A}_{(23),[1,1,1]}-\mathbb{A}_{(13),[1,1,1]}
\right)\\
&
= \mathbb{A}_{e,[3]} + \left(\mathbb{A}_{e,[1,2]}-\mathbb{A}_{(12),[1,2]}\right)
+ \frac{1}{3!}\left(
\mathbb{A}_{e,[1,1,1]}
-3\mathbb{A}_{(12),[1,1,1]}
+2\mathbb{A}_{(123),[1,1,1]}
\right)\\
&
= \sum\limits_{n=1}^{N}a_{n,3}
+ \left(\sum\limits_{n=1}^{N}a_{n,1}\right)\left(\sum\limits_{n=1}^{N}a_{n,2}\right)
- \sum\limits_{n=1}^{N}a_{n,1}a_{n,2}\\
&
+ \frac{1}{6}\left(\sum\limits_{n=1}^{N}a_{n,1}\right)^3
- \frac{1}{2}\left(\sum\limits_{n=1}^{N}a_{n,1}\right)
\left(\sum\limits_{n=1}^{N}a_{n,1}^2\right)
+ \frac{1}{3}\sum\limits_{n=1}^{N}a_{n,1}^3.
\end{align*}

\subsection{Bell polynomial} % (fold)
\label{sec:bell_polynomial}

% Furthermore, let $f(\lambda) = 1+\sum\limits_{k=1}^{\infty} a_k \lambda^k$,
Let $f(\lambda) = 1+\sum\limits_{k=1}^{\infty} a_k \lambda^k$,
and let the $N$-th power of $f(\lambda)$, $(f(\lambda))^N$, to be $1+\sum\limits_{k=1}^{\infty} X_k\lambda^k$, then apply finite product case above to $(f(\lambda))^N$ with $a_{n,k} := a_k$ for $1\le n\le N$, we have
$$
X_k=
\sum\limits_{L \vdash k}
\frac{1}{C_{stb}(L)}
\sum\limits_{\sigma\in S_{\ell(L)}}
sign(\sigma)\mathbb{A}_{\sigma,L}.$$

On the other hand, the $k$-th power of a power series can be expressed by the Bell polynomials, see \cite{Advanced_Combinatorics}.
The \textbf{Ordinary Bell polynomials} are defined as follows,
$$\hat{B}_{n,k}(x_1, x_2,\cdots, x_{n-k+1}):=\sum\frac{k!}{j_1!j_2!\cdots j_{n-k+1}!}x_1^{j_1}x_2^{j_2}\cdots x_{n-k+1}^{j_{n-k+1}}$$
where the sum runs over integers $j_1,j_2,\cdots j_{n-k+1}\ge 0$ satisfying
$$\sum\limits_{i=1}^{n-k+1}j_i = k$$
and
$$\sum\limits_{i=1}^{n-k+1}ij_i = n.$$
Then
$$\left(\sum\limits_{j=1}^{\infty}x_jt^j\right)^k
=\sum\limits_{n=k}^{\infty}
\hat{B}_{n,k}(x_1, x_2,\cdots, x_{n-k+1})t^n.$$

Apply this to our $f(\lambda)$ by calculating 
$(\lambda f(\lambda))^N
= (\lambda (1+\sum\limits_{k=1}^{\infty} a_k \lambda^k))^N
= (\lambda +\sum\limits_{k=1}^{\infty} a_k \lambda^{k+1})^N
= \lambda^N+\sum\limits_{k=1}^{\infty} X_k\lambda^{k+N}
$, and we obtain a new perspective of Ordinary Bell polynomials.
\begin{theorem}
\label{Bell}
Let $k, N\in \mathbb{N}$.
$$\hat{B}_{k+N,N}(a_0, a_1,\cdots, a_k)
=
\sum\limits_{L \vdash k}
\frac{1}{C_{stb}(L)}
\sum\limits_{\sigma\in S_{\ell(L)}}
sign(\sigma)\mathbb{A}_{\sigma,L}(\{a_{n,k}\}_{n,k=1}^\infty),
$$
where $a_0:=1$ and 
$a_{n,k}:=
\begin{cases}
a_k  &n<N\\
0  &\text{otherwise}
\end{cases}$.
\end{theorem}

% Theorem \ref{Bell} recovers the Ordinary Bell polynomials $\hat{B}_{k+N,N}(x_0, x_1,\cdots, x_k)$ for $x_0=1$.
% For the case $x_0=0$, $\hat{B}_{k+N,N}(0, x_1,\cdots, x_k)=0$ by definition.
% For the case $x_0\neq 0, 1$, we use the identity
% $$\hat{B}_{k+N,N}(x_0, x_1,\cdots, x_k) = x_0^N\hat{B}_{k+N,N}(1, \frac{x_1}{x_0},\cdots, \frac{x_k}{a_0}).$$
% For $k=0$, $\hat{B}_{N,N}(x_0) = x_0^N$ by definition.
% Thus, we already covered all the case for Ordinary Bell polynomials.
We can recover Ordinary Bell polynomials for all the cases. In particular,
Theorem \ref{Bell} recovers the Ordinary Bell polynomials $\hat{B}_{k+N,N}(x_0, x_1,\cdots, x_k)$ for $x_0=1$.
For the case $x_0=0$, $\hat{B}_{k+N,N}(0, x_1,\cdots, x_k)=0$ by definition.
For the case $x_0\neq 0, 1$, we use the identity
$\hat{B}_{k+N,N}(x_0, x_1,\cdots, x_k) = x_0^N\hat{B}_{k+N,N}(1, \frac{x_1}{x_0},\cdots, \frac{x_k}{a_0}).$
For $k=0$, $\hat{B}_{N,N}(x_0) = x_0^N$ by definition.

We demonstrate some calculations for low-degree terms of Ordinary Bell polynomials,
\begin{align*}
\hat{B}_{1+N,N}(a_0, a_1)
&= Na_1\\
\hat{B}_{2+N,N}(a_0, a_1, a_2)
&= Na_2
+ \frac{1}{2}(Na_1)^2
- \frac{1}{2}Na_1^2
= Na_2
+ \frac{N(N-1)}{2}a_1^2\\
\hat{B}_{3+N,N}(a_0, a_1, a_2, a_3)
&
= Na_3
+ (Na_1)(Na_2)
- Na_1a_2
+ \frac{1}{6}(Na_1)^3
- \frac{1}{2}(Na_1)
(Na_1^2)
+ \frac{1}{3}Na_1^3\\
&
= Na_3
+ N(N-1)a_1a_2
+ \frac{N(N-1)(N-2)}{6}a_1^3\\
\hat{B}_{4+N,N}(a_0, a_1, a_2, a_3, a_4)
&
= Na_4
+ (Na_1)(Na_3)
- Na_1a_3
+ \frac{1}{2}\left(
(Na_2)^2
- Na_2^2
\right)\\
&
+ \frac{1}{2}\left(
(Na_1)(Na_1)(Na_2)
+ 2(Na_1^2a_2)
- (Na_1^2)(Na_2)
- 2(Na_1)(Na_1a_2)
\right)\\
&
+ \frac{1}{4!}\left(
(Na_1)^4
- 6 (Na_1)^2(Na_1^2)
+ 8 (Na_1)(Na_1^3)
- 6 Na_1^4
+ 3 (Na_1^2)^2
\right)\\
&
= Na_4
+ N(N-1)a_1a_3
+ \frac{N(N-1)}{2}a_2^2\\
&
+ \frac{N(N-1)(N-2)}{2}a_1^2a_2
+ \frac{N(N-1)(N-2)(N-3)}{4!}a_1^4.
\end{align*}

In abstract, we have mentioned that we also recover \textbf{Binomial Theorem} and \textbf{Multinomial Theorem}.
In fact, these two theorems are just the special cases for bell polynomials.

Consider $(1+a_1+\cdots+a_\alpha)^N$ for some fixed $\alpha$. It is easy to see that this is an special case for the bell polynomial case
\[
(1+a_1\lambda +\cdots+a_\alpha\lambda^\alpha+\cdots)^N
=\sum\limits_{k=0}^{\infty} X_k \lambda^k
\]
by substituting the $\lambda$ by $1$ and truncating $a_k =0$ for all $k >\alpha$.

So, we have $(1+a_1+\cdots+a_\alpha)^N=
\sum\limits_{k=0}^{\alpha N} X_k$.
% Thus, we have the following new perspective of the Multinomial Theorem for double counting.
Thus, we have the following as a double counting of the Multinomial Theorem.
\begin{theorem}
\label{Multinomial}
(Alternating expression of Multinomial theorem)
Let $k, N\in \mathbb{N}$.
$$(1+a_1+\cdots+a_\alpha)^N
=
1+\sum\limits_{k=1}^{\alpha N}
\sum\limits_{L \vdash k}
\frac{1}{C_{stb}(L)}
\sum\limits_{\sigma\in S_{\ell(L)}}
sign(\sigma)\mathbb{A}_{\sigma,L}(\{a_{n,k}\}_{n,k=1}^\infty),
$$
where $a_0:=1$ and 
$a_{n,k}:=
\begin{cases}
a_k  &n<N \text{ and } k \le \alpha\\
0  &\text{otherwise}
\end{cases}$.
\end{theorem}
% subsection bell_polynomial (end)

% section some_notable_application_in_finite_product (end)

\section{Convergence} % (fold)
\label{sec:convergence}

Recall that we assume that each power series $1+\sum\limits_{k=1}^{\infty} a_{n,k}\lambda^k$ is analytic
and the infinite product
$\prod\limits_{n=1}^{\infty}(1+\sum\limits_{k=1}^{\infty} a_{n,k}\lambda^k)$ uniformly converges.
In this section, we do \textbf{not} assume the convergence of the forms in $S$, see Equation (\ref{generating_set}).

As the idea in Section \ref{sec:some_notable_applications_in_finite_product_cases}, 
we denote the truncated notation 
$$
\epsilon_{n,k}^{(N)}=
\begin{cases}
1 & n<N\\
0 & \text{otherwise.}
\end{cases}
$$
Rewrite the products to the power series in $\lambda$ again,
$$\prod\limits_{n=1}^{N}(1+\sum\limits_{k=1}^{\infty} a_{n,k}\lambda^k)
=\prod\limits_{n=1}^{\infty}(1+\sum\limits_{k=1}^{\infty} a_{n,k}\epsilon_{n,k}^{(N)}\lambda^k)
=1+\sum\limits_{k=1}^{\infty} X_k^{(N)}\lambda^k,$$
$$\prod\limits_{n=1}^{\infty}(1+\sum\limits_{k=1}^{\infty} a_{n,k}\lambda^k)
=1+\sum\limits_{k=1}^{\infty} X_k\lambda^k.$$
Then we have 
$X_k^{(N)}=
\sum\limits_{L \vdash k}\frac{1}{C_{stb}(L)}
\sum\limits_{\sigma\in S_{\ell(L)}}
sign(\sigma)\mathbb{A}_{\sigma,L}(\{a_{n,k}\epsilon_{n,k}^{(N)}\}_{n,k=1}^\infty)$.
Furthermore, since
$\prod\limits_{n=1}^{N}(1+\sum\limits_{k=1}^{\infty} a_{n,k}\lambda^k)$
uniformly converges to
$\prod\limits_{n=1}^{\infty}(1+\sum\limits_{k=1}^{\infty} a_{n,k}\lambda^k)$,
we get the convergence of the coefficient $\lim_{N\to\infty}X_k^{(N)}=X_k$ for all $k\in\mathbb{N}$.
Thus, we have the following theorem.

\begin{theorem}
\label{convergence_main}
For $k\in\mathbb{N}$.
$$
X_k = \lim_{N\to\infty}\sum\limits_{L \vdash k}
\frac{1}{C_{stb}(L)}
\sum\limits_{\sigma\in S_{\ell(L)}}
sign(\sigma)\mathbb{A}_{\sigma,L}(\{a_{n,k}\epsilon_{n,k}^{(N)}\}_{n,k=1}^\infty).
$$
\end{theorem}
We provide some low-degree terms by applying Theorem \ref{convergence_main},
\begin{align*}
X_1
&= \lim_{N\to\infty}\sum\limits_{n=1}^{N}a_{n,1}= \sum\limits_{n=1}^{\infty}a_{n,1}\\
X_2
&= \lim_{N\to\infty}\left[\sum\limits_{n=1}^{N}a_{n,2}
+ \frac{1}{2}\left(\sum\limits_{n=1}^{N}a_{n,1}\right)^2
- \frac{1}{2}\sum\limits_{n=1}^{N}a_{n,1}^2\right]\\
X_3
&
= \lim_{N\to\infty}\left[\sum\limits_{n=1}^{N}a_{n,3}
+ \left(\sum\limits_{n=1}^{N}a_{n,1}\right)\left(\sum\limits_{n=1}^{N}a_{n,2}\right)
- \sum\limits_{n=1}^{N}a_{n,1}a_{n,2}\right.\\
&
\left.
+ \frac{1}{6}\left(\sum\limits_{n=1}^{N}a_{n,1}\right)^3
- \frac{1}{2}\left(\sum\limits_{n=1}^{N}a_{n,1}\right)
\left(\sum\limits_{n=1}^{N}a_{n,1}^2\right)
+ \frac{1}{3}\sum\limits_{n=1}^{N}a_{n,1}^3\right].
\end{align*}

\begin{example}
We give an example that the Theorem \ref{finalcoeff} fails, but the Theorem \ref{convergence_main} works.
Let
$$f_n (x):= e^{\frac{(-1)^n}{\sqrt[4]{n}}x}
=1+\frac{(-1)^n}{\sqrt[4]{n}}x+\frac{1}{2\sqrt{n}}x^2+O(x^3)$$
and $M:=\sum\limits_{n=1}^{\infty}\frac{(-1)^n}{\sqrt[4]{n}}\in\mathbb{R}$.
It is easy to see that $\prod\limits_{n=1}^N f_n$ converges to
$$g(x):=e^{Mx}=1+Mx+\frac{M^2}{2}x^2+O(x^3)$$
 as $N\to\infty$.
In this case, the $k$-th coefficient of power series of $f_n$ is
$a_{n,k}=\frac{(-1)^{nk}}{k!\sqrt[4]{n^k}}
$, and the form
$\sum\limits_{n=1}^{\infty}a_{n,2}
=\sum\limits_{n=1}^{\infty}\frac{1}{2\sqrt{n}}$ does not converge.
But
\begin{align*}
X_2
&=\lim_{N\to\infty}\left[\sum\limits_{n=1}^{N}a_{n,2}
+ \frac{1}{2}\left(\sum\limits_{n=1}^{N}a_{n,1}\right)^2
- \frac{1}{2}\sum\limits_{n=1}^{N}a_{n,1}^2\right]\\
&=\lim_{N\to\infty}\left[\sum\limits_{n=1}^{N}\frac{1}{2\sqrt{n}}
+ \frac{1}{2}\left(\sum\limits_{n=1}^{N}a_{n,1}\right)^2
- \frac{1}{2}\sum\limits_{n=1}^{N}\frac{1}{\sqrt{n}}\right]\\
&=\lim_{N\to\infty}
\frac{1}{2}\left(\sum\limits_{n=1}^{N}a_{n,1}\right)^2\\
&=\frac{M^2}{2}
\end{align*}
converges.

\end{example}

% section convergence (end)

\noindent
\addcontentsline{toc}{section}{Acknowledgements}
\textbf{Acknowledgements.}
% The work of Zhong-Tang Wu is supported by the MOST research grants under the budget numbers 110-2112-M-029-003, 108-2112-M-029-006-MY3. The work of Kui-Yo Chen is partially supported by the MOST research grants under the budget numbers 111-2112-M-029-007, 111-2115-M-006-005-MY2 and 110-2115-M-006-007.

On the other hand, the authors would like to thank
Chuan-Tsung Chan,
Eugene Z. Xia,
Keng-Hung Steven Lin,
Ông Bín-tsê,
Jun Yi Guo,
and
Chung-Chuan Chen
for their assistance.
Especially, the authors want to thank
Chuan-Tsung Chan,
Eugene Z. Xia,
Keng-Hung Steven Lin
and
Ông Bín-tsê
for their useful suggestions for not only math parts but also non-math parts in the whole article.

% This work is partially supproted by MOST 104-2115-M-001-007 and 105-2115-M-001-003. The author would like to thank Yi-Chiuan Chen, Chung-Chuan Chen, Chiun-Chuan Chen and Chun-Wei Lee for their assistance.
% Specially, the main ideal of this paper is inspired by the following consideration proposed by Yi-Chiuan Chen, the consideration is ``Let us consider some example of hypercyclic weighted translation operators when a is aperiodic.'', so as the author considers more examples, the general case becomes more clear.

\bibliographystyle{abbrv}
\bibliography{INFINITE_PRODUCT_OF_POWER_SERIES} 
% \phantomsection
% \addcontentsline{toc}{chapter}{Bibliography}

% delete .bbl file and try again.
\begingroup
\global\let\savedifeof=\ifeof
\def\ifeof#1{\global\let\ifeof=\savedifeof\iftrue}%

\endgroup

\vspace{.1in}
\end{document}